\documentclass[a4paper,12pt,reqno]{amsart}
\usepackage{amssymb}
\usepackage{amsfonts}
\usepackage{amsmath}
\usepackage{mathrsfs}
\DeclareMathOperator{\ch}{char} \DeclareMathOperator{\ann}{ann}
\DeclareMathOperator{\aut}{Aut}

\newcommand{\K}{{\mathbb K}}
\newcommand{\C}{{\mathbb C}}

\newcommand{\Z}{{\mathbb Z}}
\newcommand{\N}{{\mathbb N}}

\newcommand{\Y}{{\mathcal Y}}
\newcommand{\W}{{\mathcal W}}

\newtheorem{theorem}{Theorem}[section]
\newtheorem{lemma}[theorem]{Lemma}
\newtheorem{cor}[theorem]{Corollary}
\newtheorem{prop}[theorem]{Proposition}

\theoremstyle{definition}
\newtheorem{defns}[theorem]{Definitions}
\newtheorem{notn}[theorem]{Notation}
\newtheorem{ex}[theorem]{Example}
\newtheorem{rmk}[theorem]{Remark}

\begin{document}

\title{Prime spectra of ambiskew polynomial rings}

\author{Christopher D. Fish}
\address{School of Mathematics and Statistics\\
University of Sheffield\\
Hicks Building\\
Sheffield S3~7RH\\
UK}
\email{christopher.fish@cantab.net}
\author{David A. Jordan}
\address{School of Mathematics and Statistics\\
University of Sheffield\\
Hicks Building\\
Sheffield S3~7RH\\
UK}
\email{d.a.jordan@sheffield.ac.uk}



\thanks{Some of the results in this paper appear in the University of Sheffield PhD thesis of the first author who acknowledges the support of the Engineering and Physical Sciences Research Council of the UK. We thank the referee for several suggestions that have significantly improved the exposition.}

\begin{abstract}
We determine sufficient criteria for the prime spectrum of an ambiskew polynomial algebra $R$ over an algebraically closed field $\K$ to be akin to those of two of the principal examples of such an algebra, namely the universal enveloping algebra $U(sl_2)$ (in characteristic $0$) and its quantization $U_q(sl_2)$ (when $q$ is not a root of unity). More precisely, we determine sufficient criteria for the prime spectrum  of $R$ to consist of $0$, the ideals $(z-\lambda)R$ for some central element $z$ of $R$ and all $\lambda\in \K$, and, for some positive integer $d$ and each positive integer $m$, $d$ height two prime ideals $P$ for which $R/P$ has Goldie rank $m$.
\end{abstract}

\maketitle

\thispagestyle{empty}

\section{Introduction}
The results of this paper are applicable to the determination of the prime ideals of certain ambiskew polynomial algebras and generalized Weyl algebras. For readers unfamiliar with these algebras, details appear at the end of this introduction. The main results of \cite{ambiskew} are simplicity criteria for an ambiskew polynomial algebra  $R$ over a field $\K$ and, in cases where $R$ is not itself simple, certain localizations and factors of $R$ including generalized Weyl algebras. Such results are applicable to the analysis of the prime spectrum of an ambiskew polynomial ring or of any ring which has an ambiskew polynomial ring as a localization.
Our aim is to prove results that can be applied together to show that, under appropriate conditions, the prime spectrum of a given algebra $R$ over an algebraically closed field $\K$ meets the following description $(*)$:
\begin{itemize}
\item $0$ is a prime ideal, \item there exists $z\in Z(R)$ (the centre of $R$) such that the height one prime ideals have the form $(z-\lambda)R$, $\lambda\in \K$, \item $(z-\lambda)R$ is maximal for all but countably many values of $\lambda$ and \item there is a positive integer $d$ such that,
 for each $m\geq 1$, $R$ has $d$ height two prime ideals $P$ for which $R/P$ has Goldie rank $m$.\end{itemize}
It is well-known that the prime spectra of the universal enveloping algebra $U(sl_2)$ (in characteristic $0$) and the universal quantized enveloping algebra $U_q(sl_2)$ (when $q$ is not a root of unity) fit the description $(*)$ with $d=1$ and $2$ respectively. These two algebras are among the main examples of ambiskew polynomial rings. They are well-understood and will serve to illustrate our results. The new application will be to certain ambiskew polynomial rings over coordinate rings of quantum tori which arise, as localizations, in our analysis of
connected quantized Weyl algebras \cite{cdfdaj}. For these algebras, which will be introduced in Example~\ref{quantum torus}, we shall see that the prime spectrum
fits the description $(*)$ with $d=2$. This is applied in \cite{cdfdaj} to the determination of the prime spectra of connected quantized Weyl algebras.

The first step in establishing $(*)$ for a domain is to identify an appropriate central element $z$ for which the localization of $R$ at $\K[z]\backslash\{0\}$ is simple. This will be done in Section~2 using the notion of a  Casimir element for an ambiskew polynomial ring. When such elements exist, they are normal  but not necessarily central. \cite[Theorem 4.7]{ambiskew} is a simplicity criterion for the localization of $R$ at the powers of $z$. If $z$ is central then this localization is never simple and the appropriate localization for which to consider simplicity is at $\K[z]\backslash\{0\}$. In Proposition~\ref{zlocsimple}, we give sufficient conditions for this localization to be simple. As the localization is central, all  ideals of $R$ extend to ideals of the localization and simplicity of the localization is equivalent to the property that every non-zero ideal $R$ has non-zero intersection with $\K[z]$.
Proposition~\ref{zlocsimple2} generalizes Proposition~\ref{zlocsimple} to a
 situation where
there is a central polynomial subalgebra $\K[z,c_1,\dots,c_t]$ of $R$ for some $t\geq 0$. This general result will be applied, with $t=1$ to show that  the augmented down-up algebras of \cite{terwor} have the property that every non-zero ideal has non-zero intersection with the centre which, for these algebras, is a polynomial algebra in two indeterminates.

Having completed the first step, we proceed, in Section 3, to analyse prime spectra of the factors $R/(z-\lambda)R$ for $\lambda\in \K$. For description $(*)$ to hold we need all but countably many of these to be simple.  These factors are generalized Weyl algebras $W(A,\alpha,u)$ in the sense of \cite{vlad1} and we apply known
simplicity criteria \cite{vlad5,ambiskew} for $W(A,\alpha,u)$  to give, in Corollary~\ref{zlambdamax}, sufficient conditions, involving a positive integer parameter $m$, for $R/(z-\lambda)R$ to be simple. We then proceed to give, in Corollary~\ref{unique}, sufficient conditions for $R/(z-\lambda)R$ to have a unique non-zero prime ideal
$P/(z-\lambda)R$. In Section 4, under mild extra conditions,
 we show that the parameter $m$ is the right Goldie rank of  $R/P$ for the unique prime ideal $P/(z-\lambda)R$ of $R/(z-\lambda)R$ when it exists.

In the motivating examples arising from quantum tori, it turns out that for each $m\geq 1$, there are precisely two values of $\lambda$ for which $R/(z-\lambda)R$ is not simple and, for each of these values of $\lambda$, $R/(z-\lambda)R$ has a unique non-zero prime ideal.
For $U(sl_2)$ and the quantized enveloping algebra $U_q(sl_2)$
the exceptional maximal ideals are annihilators of finite-dimensional simple modules but this is not the case for the examples over quantum tori, where the factors are infinite-dimensional.

In the remainder of the introduction, we give some reminders of the construction and properties of ambiskew polynomial rings and generalized Weyl algebras.
\begin{defns}\label{defambi}
Let $\K$ be a field, and let $A$ be a
 $\K$-algebra. For convenience, we shall assume that $\K$ is algebraically closed. Let $\rho\in
\K\backslash\{0\}$ and let $v$ be a central element of $A$. Let
$\alpha\in\aut_{\K}A$ and let $\beta=\alpha^{-1}$. Extend $\beta$ to a $\K$-automorphism of
$A[y;\alpha]$ by setting $\beta(y)=\rho y$. There is a
$\beta$-derivation $\delta$ of $A[y;\alpha]$ such that $\delta(A)=0$ and
$\delta(y)=v$. The {\it ambiskew polynomial algebra} $R(A,\alpha,v,\rho)$ is the iterated skew polynomial algebra
$A[y;\alpha][x;\beta,\delta]$. Thus $ya=\alpha(a)y$ and $xa=\beta(a)x$  for all $a\in A$ and
$xy=\rho yx+v.$

More general versions of ambiskew polynomial algebras are considered in  \cite{ambiskew}, where $v$ need not be central and $\beta$ need not be $\alpha^{-1}$,  and \cite{downup}, where $\alpha$ need not be bijective, but here we consider only the case specified above.

If there is a central element $u\in A$ such that
$v=u-\rho\alpha(u)$ then the element
$z=xy-u=\rho(yx-\alpha(u))$ is such that
$zy=\rho yz$, $zx=\rho^{-1} xz$ and $za=az$ for all $a\in A$. Hence
$z$ is normal in $R$, i.e. $zR=Rz$, and it is central if and only if $\rho=1$.
If such an element $u$ exists then it is called a {\it splitting} element
and we say that $R$ is a {\it conformal} ambiskew polynomial algebra.
We then refer to the element
$z$ as the {\it Casimir element}
of $R$. If $\rho=1$ then $u$ and $z$ are not unique and, for any $\lambda\in \K$, can be replaced by $u+\lambda$ and $z-\lambda$ respectively.

Let
$v^{(0)}=0$ and $v^{(m)}=\sum_{l=0}^{m-1}\rho^{l}\alpha^{l}(v)$ for
$m\in\N$. In particular $v^{(1)}=v$.  Each  $v^{(m)}$ is central and
it is easily checked, by induction, that, for $m\geq 0$,
\begin{eqnarray}
\label{skewcomm} xy^{m}-\rho^{m}y^{m}x & = & v^{(m)} y^{m-1}\;\text{ and}
\\
 x^{m}y-\rho^{m}yx^{m} & = & x^{m-1}v^{(m)} = \alpha^{1-m}(v^{(m)})x^{m-1}.
\end{eqnarray}
If $u$ is a splitting element then $v^{(m)}=u-\rho^m\alpha^m(u)$.

\end{defns}

\begin{defns}\label{GWA}\cite{vlad1} Let $A$ be a ring, $\alpha$ be an automorphism of $A$, with inverse $\beta$, and $u$ be a central element of $A$.
The \emph{generalized Weyl algebra} $W(A,\alpha,u)$ is generated, as a ring extension of $A$, by $X$ and $Y$ subject to the relations $YX=\alpha(u)$, $XY=u$ and,  for all $a\in A$, $Ya=\alpha(a)Y$ and $Xa=\beta(a)X$.    Here $A$ and $\alpha$ will be a $\K$-algebra and a $\K$-automorphism respectively. If $R$ is the conformal ambiskew polynomial ring $R(A,\alpha,u-\alpha(u),1)$, with Casimir element $z$, then we may identify
$W(A,\alpha,u)$ with $R/zR$, $X$ with $x+zR$ and $Y$ with $y+zR$.

The algebra $W=W(A,\alpha,u)$ has a $\Z$-grading in which $W_0=A$ and, for $i>0$,  $W_i=AY^i$ and $W_{-i}=AX^i$. If $A$ is a domain then, by the $\Z$-grading, so too, for each $\lambda\in \K$,  is $W(A,\alpha,u+\lambda)\simeq R/(z-\lambda)R$. Hence, if $A$ is a domain, $(z-\lambda)R$ is a completely prime ideal of $R$ for all $\lambda\in \K$.

It is easy to check inductively that, for all $m\geq 1$,
\[X^{m}Y^{m}=\displaystyle\prod_{i=0}^{m-1}\alpha^{-i}(u)\text{ and }
Y^{m}X^{m}=\displaystyle\prod_{i=1}^{m}  \alpha^{i}(u).\]
As observed in \cite[Notation 5.3]{ambiskew}, the isomorphic skew Laurent polynomial rings
$A[Y^{\pm 1};\alpha]$ and $A[X^{\pm 1};\alpha^{-1}]$ are the localizations of $W$ at the Ore sets $\{Y^i: i\geq 1\}$ and $\{X^i: i\geq 1\}$  respectively.

\end{defns}

\begin{rmk}
The numbering of results in this version is different to that in the previous version cited in \cite{cdfdaj}. The references in \cite{cdfdaj} to  Example 3.12 and Corollary 4.7 would, in the current numbering, be to Example~\ref{spectrumSC}  and Corollary~\ref{GrankS}  respectively.
\end{rmk}

\section{Simple central localizations}
The following lemma, which in the Noetherian case is an immediate consequence of \cite[2.1.16(vi)]{McCR}, is a generalization of \cite[Lemma 3.1]{ambiskew}.
\begin{lemma}\label{yIy}
Let $B$ be a ring, let $y$ be a regular element of $B$ such that ${\Y}:=\{y^i\}_{i\geq1}$ is a right and left Ore set and let $\mathcal{Z}$ be a multiplicatively closed set of central elements of $B$. Let $\W=\{y^iz:i\geq 1, z\in \mathcal{Z}\}$, which a right and left Ore set,
and let $C=B_{\W}$ be the localization of $B$ at ${\W}$. If $C$ is simple and $I$ is a non-zero ideal of $B$ then $y^sz\in I$ for some $s\geq 0$ and some $z\in \mathcal{Z}$.
\end{lemma}
\begin{proof}
Note that $C=(B_\mathcal{Z})_\mathcal{Y}$. It follows easily from the centrality of $\mathcal{Z}$ that $IB_\mathcal{Z}$ is an ideal of $B_\mathcal{Z}$. By \cite[Lemma 3.1]{ambiskew}, $y^s\in IB_\mathcal{Z}$  for some $s\geq 0$. By \cite[2.1.16(iv)]{McCR}, $y^sz\in I$ for some $z\in \mathcal{Z}$.
\end{proof}

\begin{prop}
\label{zlocsimple}
Let $R$ be a conformal ambiskew polynomial ring of the form $R(A,\alpha,v,1)$ where $A$ is a $\K$-algebra and $v\in A$ is central. Let $u$ be a splitting element, with  corresponding Casimir element $z=xy-u$, and let $\mathcal{Z}$ be the multiplicatively closed set of central elements $\K[z]\backslash\{0\}$. Suppose that $A[y^{\pm 1};\alpha]$ is simple and that $Z(A[y^{\pm 1};\alpha])=\K$. Then $R_\mathcal{Z}$ is simple if and only if, for all $m\geq 1$, there exists a non-zero polynomial $p_m(X)\in \K[X]$ such that $p_m(u)\in v^{(m)}A$.
\end{prop}
\begin{proof}
Assume that
for all $m\geq 1 $, there exists a non-zero polynomial $p_m(X)\in \K[X]$ such that $p_m(u)\in v^{(m)}A$. Let $\mathcal{Y}=\{y^i\}_{i\geq 0}$ and let $\mathcal{W}:=\{y^mq(z):m\geq 1, q(z)\in \mathcal{Z}\}$.  By the centrality of $\mathcal{Z}$, $\mathcal{W}$ is a right and left Ore set in $R$ and
$R_\mathcal{W}=(R_\mathcal{Y})_\mathcal{Z}=(R_\mathcal{Z})_\mathcal{Y}$. We first show that $R_\mathcal{W}$ is simple.
 The argument in \cite[1.5]{itskew}, where $A$ is commutative, is valid more generally and shows that $\mathcal{Y}$ is a right and left Ore set in $R$ and $R_\mathcal{Y}=A[y^{\pm1};\alpha][z]$.
As  $A[y^{\pm 1};\alpha]$ is simple and $Z(A[y^{\pm 1};\alpha])=\K$, it follows from \cite[Lemma 9.6.9]{McCR}, with $V=\K[z]$, that $R_\mathcal{W}$ is simple.

Now suppose that $R_\mathcal{Z}$ is not simple, let $M\neq 0$ be a maximal ideal of $R_\mathcal{Z}$ and let $P=M\cap R$. Then $\mathcal{Z}\cap P=\emptyset$ and
 $P\neq 0$. Using the centrality of $\mathcal{Z}$, it is easy to check that $P$ is a prime ideal of $R$ and that
 $q(z)$ is regular modulo $P$ for all $q(z)\in \mathcal{Z}$.   By Lemma~\ref{yIy} and the simplicity of $R_\mathcal{W}$, $y^jq(z)\in P$ for some $j\geq 0$ and some $q(z)\in \mathcal{Z}$. Hence
  $y^j\in P$ and there exists a minimal $m\geq 0$ such that $y^m\in P$.
	
Suppose  that $m\geq 1$.  By assumption, there exists a non-zero polynomial $p_m(X)\in \K[X]$ such that $p_m(u) \in v^{(m)}A$. By
\eqref{skewcomm},
$v^{(m)}y^{m-1} \in P$ whence $v^{(m)}Ay^{m-1} \subset P$ and $p_m(u)y^{m-1} \in P$. As $u$ and $xy$ commute, $(-z)^i=(u-xy)^i\equiv u^i\bmod Ry$ for $i\geq 0$ and hence $p_m(-z)\equiv p_m(u)\bmod Ry$. Therefore $p_m(-z)y^{m-1}\equiv p_m(u)y^{m-1}\bmod Ry^m$ and so, as $p_m(u)y^{m-1}\in P$ and $y^m\in P$,  we see that $p_m(-z)y^{m-1}\in P$. The regularity of $p_m(-z)$ modulo $P$ then gives that $y^{m-1}\in P$, contradicting the minimality of $m$. Thus $m=0$, $1\in P$ and $M=W$. This contradiction shows that $R_\mathcal{Z}$ is simple.

Conversely, suppose that $R_\mathcal{Z}$ is simple.  Let $m\geq 1$.
As in the proof of
\cite[Lemma 4.1]{ambiskew}, let $J$ be the $\K$-subspace of $R$ spanned by the elements of the form
$x^iay^j$ where $i>0$ or $j\geq m$ or $a\in v^{(m)}A$. Then
$J$ is a
right ideal of $R$ and $I:=\ann_{R}(R/J)$ is
an  ideal of $R$ contained in $J$ and containing $y^m$. Note that
$J\cap A=v^{(m)}A$. As $\mathcal{Z}$ is central, $IR_\mathcal{Z}$ is a non-zero ideal of the simple ring $R_\mathcal{Z}$ so, by \cite[Proposition 2.1.16(iv)]{McCR},  it follows that $p_m(-z)\in I$ for some  non-zero polynomial $p_m(X)\in \K[X]$.  Thus $p_m(u-xy)\in J$ and, as $x\in J$ and $uxy=xyu\in
J$, it follows that $p_m(u)\in J\cap A=v^{(m)}A$.
\end{proof}

\begin{rmk}\label{altcon}
In Proposition~\ref{zlocsimple}, the hypotheses  that   $Z(A[y^{\pm 1};\alpha])=\K$ and $A[y^{\pm 1};\alpha]$ is simple can be rephrased in terms of the base ring $A$.  Using \cite[Theorem 1.8.5]{McCR}, it is easy to check that these conditions are equivalent to the following three conditions:
\begin{enumerate}
\item $A$ is $\alpha$-simple;
\item $\alpha^n$ is outer for all positive integers $n$;
\item
$\{a\in Z(A):\alpha(a)=a\}=\K$.
\end{enumerate}
\end{rmk}

\begin{cor}\label{UFD}
Let $R$ be a conformal ambiskew polynomial ring of the form $R(A,\alpha,u-\alpha(u),1)$, where $u$ is central and the $\K$-algebra $A$ is a domain such that $A[y^{\pm 1},\alpha]$ is simple, $Z(A[y^{\pm 1};\alpha])=\K$ and for all $m\geq 1$, there exists a non-zero polynomial $p_m(X)\in \K[X]$ such that $p_m(u)\in v^{(m)}A$.  Then the height one prime ideals of $R$ are the ideals of the form $(z-\lambda)R$, $\lambda\in \K$, where $z$ is the Casimir element $xy-u$. Consequently $R$ is a UFD (in the sense of \cite{Chatt}).
\end{cor}
\begin{proof}
By Proposition~\ref{zlocsimple}, the localization of $R$ at $\K[z]\backslash\{0\}$ is simple.   As $z$ is central and $\K$ is algebraically closed, it follows from Lemma~\ref{yIy}, with $y=1$, that every non-zero prime ideal  of $R$ contains $(z-\lambda)R$ for some $\lambda\in \K$. As observed in \ref{GWA}, $(z-\lambda)R$ is completely prime for each $\lambda\in \K$ so the result follows.
\end{proof}

We can apply Proposition 2.2 to obtain alternative proofs of known results, including \cite[Theorem 4.6]{itskew}   and \cite[Theorem 3.2]{Smith}.
\begin{cor}\label{uuq}
Let $A$ be either (i) $\K[t]$ where $\alpha(t)=t+\mu$ for some $\mu\in \K^*$ and $\ch(\K)=0$ or
(ii) $\K[t^{\pm 1}]$ where $\alpha(t)=qt$ for
some $q\in \K^*$ that is not a root of unity. Let $u\in A\backslash{\K}$ and let $R=R(A,\alpha,u-\alpha(u),1)$. Then the height one prime ideals of $R$ are the ideals of the form $(z-\lambda)R$, $\lambda\in \K$, where $z$ is the Casimir element $xy-u$, and every non-zero ideal of $R$ has non-zero intersection with $\K[z]$.
\end{cor}
\begin{proof}
It is well-known that in both cases $A$ is $\alpha$-simple and it is clear that Conditions (ii) and (iii) of Remark~\ref{altcon} hold. Hence $A[y^{\pm 1},\alpha]$ is simple and $Z(A[y^{\pm 1};\alpha])=\K$. For $m\geq 1$, $v^{(m)}=u-\alpha^m(u)\neq 0$ so the $\K$-algebra $A/v^{(m)}A$ is finite-dimensional, $u+v^{(m)}A$ is algebraic over $\K$ and there exists a non-zero polynomial $p_m(X)\in \K[X]$ such that $p_m(u)\in v^{(m)}A$. By Corollary~\ref{UFD} the height one prime ideals of $R$ are the ideals of the form $(z-\lambda)R$, $\lambda\in \K$, and, $R$ being noetherian, it follows that every non-zero ideal of $R$ has non-zero intersection with $\K[z]$.
\end{proof}

In the  following two examples we give details of the best known examples of cases (i) and (ii) of Corollary~\ref{uuq}. They are included to illustrate our results rather than to advance understanding of the examples. We shall need to know the values of the elements $v^{(m)}$, $m\geq 1$.

\begin{ex}\label{usltwo}
Assume that $\ch(\K)=0$. Let $A$ be the polynomial algebra $\K[t]$ and let $\alpha$ be the $\K$-automorphism of $A$ such that $\alpha(t)=t+2$.  Let $\rho=1$ and let $u=\frac{-1}{4}(t-1)^2$, so that $v=t$. Then $R(A,\alpha,v,1)$ is the enveloping algebra
$U(sl_2)$, in which $x, y$ and $t$ are usually written $e, f$ and $h$. In the notation of Definitions~\ref{defambi}, the Casimir element $z$  is $\frac{1}{4}(\Omega+1)$, where $\Omega$ is the usual Casimir element as, for example, in \cite{dix}. For
$m\geq 1$, $v^{(m)}=m(t+m-1)$. In accordance with Proposition~\ref{zlocsimple} and Corollary~\ref{uuq}, the localization of $R$ at $\K[z]\backslash\{0\}$ is simple.
\end{ex}

\begin{ex}\label{usltwoq} Let $q\in \K$ and suppose that $q$ is not a root of unity. Let $A$ be the Laurent polynomial algebra $\K[t^{\pm 1}]$ and let $\alpha$ be the $\K$-automorphism of $A$ such that $\alpha(t)=q^2t$. Again, it is well-known that $A$ is $\alpha$-simple. Let $\rho=1$ and let \[u=-(q^{-1}t+qt^{-1})/(q-q^{-1})^2\text{ and that }v=u-\alpha(u)=(t-t^{-1})/(q-q^{-1}).\] Here  $R(A,\alpha,v,1)$ is the quantum enveloping algebra
$U_q(sl_2)$, for example, see \cite[Chapter I.3]{BGl}, where, as usual, $x, y$ and $t$ are written $E, F$ and $K$. The Casimir element $z$ is \[xy+(q^{-1}t+qt^{-1})/(q-q^{-1})^2\] and, for $m\geq 1$, \[v^{(m)}=((q^{2m-1}-q^{-1})t+(q^{1-2m}-q)t^{-1})/(q-q^{-1})^2.\]  In accordance with Proposition~\ref{zlocsimple}  and Corollary~\ref{uuq}, the localization of $R$ at $\K[z]\backslash\{0\}$ is simple. Note that the  version of $U_q(sl_2)$  considered in \cite[Example 2.3]{itskew} is different to the now established one considered here.
\end{ex}

In the next example, which occurs as a localization of a connected quantized Weyl algebra in \cite{cdfdaj}, $A$ is noncommutative (if $p\neq 1$) and the results of \cite{itskew,htone} on height one prime ideals do not apply.

\begin{ex}\label{quantum torus}
Let $p$ be an odd positive integer and let $q\in \K^*$. Suppose that $q$ is not a root of unity. Let $A$ be the quantum torus with generators $z_i^{\pm 1}$, $1\leq i\leq p$, subject to the relations
$z_iz_j=q_{ij}z_jz_i$ for $1\leq j<i\leq p$, where, for $i>j$,
\[q_{ij}=\begin{cases}
1\text{ if }i\text{ is odd or if }i\text{ and }j\text{ are both even,}\\q^{-1}\text{ if }i\text{ is even and }j\text{ is odd.}\end{cases}\]
 Note that $z_p$ is central in $A$.
 Let $\alpha$ be the $\K$-automorphism of $A$ such that, for $1\leq i\leq p$,
 $\alpha(z_i)=z_i$ if $i$ is even and $\alpha(z_i)=q^{-1}z_i$ if $i$ is odd. The skew Laurent polynomial ring
$S=A[y^{\pm 1};\alpha]$ is a quantum torus in $p+1$ generators $z_i^{\pm 1}$, $1\leq i\leq p+1$, where $z_{p+1}=y$. It follows from \cite[Proposition 1.3]{McCP}, that $S$ is simple
and has centre $\K$. See \cite[Lemma 3.7]{cdfdaj} for more detail.

Let \[v=(1-q)(q^\frac{p-1}{2}z_{p}^{-1} - z_{p})\in Z(A)\] and observe that \[v=u-\alpha(u)\text{, where }
    u=q^{\frac{p-1}{2}}z_{p}^{-1}+qz_{p}.\] Thus $R:=R(A,\alpha,v,1)$ is conformal with Casimir element $z=xy-u$.
Let $m\geq 1$. Then \[v^{(m)}=u-\alpha^m(u)=(1-q^m)(q^{\frac{p-1}{2}}z_{p}^{-1}-q^{1-m}z_{p})\] so
    \[z_{p}^2\equiv q^{\frac{p+2m-3}{2}}\bmod{(v^{(m)}A)}\text{ and }z_{p}^{-2}\equiv q^{-\frac{p+2m-3}{2}}\bmod{(v^{(m)}A)}.\]
    Hence \begin{eqnarray*}
    u^2&=&q^{p-1}z_{p}^{-2}+2q^{\frac{p+1}{2}}+q^2z_{p}^{2}\\
    &\equiv&q^{\frac{p+1}{2}}(q^{-m}+2+q^m)\mod v^{(m)}A.
    \end{eqnarray*}
    Thus $p_m(u)\in v^{(m)}A$ where $p_m(X)=X^2-\sigma$ and $\sigma=q^{\frac{p+1}{2}}(q^{-m}+2+q^m)$.
    By Proposition~\ref{zlocsimple}, every non-zero prime ideal of $R$ has non-zero intersection with $\K[z]$ and, by Corollary~\ref{UFD}, every height one prime ideal has the form $(z-\lambda)R$, $\lambda\in \K$.
\end{ex}

The next result is a generalization of Proposition~\ref{zlocsimple}, which is the case $t=0$, and is applicable to other algebras in which every ideal intersects the centre non-trivially.

\begin{prop}
\label{zlocsimple2}
Let $B$ be a $\K$-algebra with a $\K$-automorphism $\alpha$ such that $B[y^{\pm 1};\alpha]$ is simple and $Z(B[y^{\pm 1};\alpha])=\K$.
Let $t\geq 0$ be an integer and let $A$ be the polynomial algebra $B[c_1,\dots,c_t]$ in $t$ algebraically independent commuting indeterminants. Extend $\alpha$ to a $\K$-automorphism of $A$ by setting $\alpha(c_i)=c_i$ for $1\leq i\leq t$. Let $u\in A$, let $v=u-\alpha(u)$ and, in the conformal ambiskew polynomial ring $R=R(A,\alpha,v,1)$, let $z$ be the Casimir element $xy-u$.

{\rm (i)} $Z(A[y^{\pm 1};\alpha])=\K[c_1,\dots,c_t]$ and $Z(R)$ is the polynomial algebra $\K[z,c_1,\dots,c_t]$.

{\rm (ii)} Let $\mathcal{Z}=Z(R)\backslash{0}$. The localization $R_\mathcal{Z}$ is simple if and only if, for all $m\geq 1$, there exists a non-zero polynomial $p_m(X,X_1,\dots,X_t)\in \K[X,X_1,\dots,X_t]$ such that $p_m(u,c_1,\dots,c_t)\in v^{(m)}A$.
\end{prop}
\begin{proof}
(i) is straightforward.

(ii) We adapt the proof of Proposition~\ref{zlocsimple} with $\mathcal{Y}=\{y^i\}_{i\geq 0}$,
$R_\mathcal{Y}=A[y^{\pm1};\alpha][z]=B[y^{\pm1};\alpha][z,c_1,\dots,c_t]$, $\mathcal{W}=\{y^mq(z,c_1,\dots,c_t):m\geq 1, q(z,c_1,\dots,c_t)\in \mathcal{Z}\}$  and
$R_\mathcal{W}=(R_\mathcal{Y})_\mathcal{Z}=(R_\mathcal{Z})_\mathcal{Y}$, which is simple.

Assume that, for all $m\geq 1$, there exists a non-zero polynomial $p_m(X,X_1,\dots,X_t)\in \K[X,X_1,\dots,X_t]$ such that $p_m(u,c_1,\dots,c_t)\in v^{(m)}A$. Suppose that $R_\mathcal{Z}$ is not simple, let $M\neq 0$ be a maximal ideal of $R_\mathcal{Z}$ and let $P=M\cap R$. Then $\mathcal{Z}\cap P=\emptyset$,
 $P\neq 0$ and, using the centrality of $\mathcal{Z}$, it is easy to check that $P$ is a prime ideal of $R$ and that $q(z,c_1,\dots,c_t)$ is regular modulo $P$ for all
$q(z,c_1,\dots,c_t)\in \mathcal{Z}$.   By Lemma~\ref{yIy} and the simplicity of $R_\mathcal{W}$, $y^jq(z,c_1,\dots,c_t)\in P$ for some $j\geq 0$ and some $q(z,c_1,\dots,c_t)\in \mathcal{Z}$. Hence $y^j\in P$. Let $m\geq 0$ be is minimal such that $y^m\in P$. As $P$ is proper, $m\geq 1$.  By assumption, there exists a non-zero polynomial $p_m(X,X_1,\dots,X_t)\in \K[X,X_1,\dots,X_t]$ such that $p_m(u,c_1,\dots,c_t) \in v^{(m)}A$. As in the proof of Proposition~\ref{zlocsimple},
$p_m(u,c_1,\dots,c_t)y^{m-1} \in P$,  $p_m(-z,c_1,\dots,c_t)\equiv p_m(u,c_1,\dots,c_t)\bmod Ry$, $p_m(-z,c_1,\dots,c_t)y^{m-1}\equiv p_m(u,c_1,\dots,c_t)y^{m-1}\bmod Ry^m$,  $p_m(-z,c_1,\dots,c_t)y^{m-1}\in P$ and $y^{m-1}\in P$, contradicting the minimality of $m$. It follows that  $R_\mathcal{Z}$ is simple.

Conversely, suppose that $R_\mathcal{Z}$ is simple.  Let $m\geq 1$.
As in the proof of Proposition~\ref{zlocsimple}, if $J$ denotes the $\K$-subspace of $R$ spanned by the elements of the form
$x^iay^j$ where $i>0$ or $j\geq m$ or $a\in v^{(m)}A$ then
$J$ is a
right ideal of $R$ and $I:=\ann_{R}(R/J)$ is
an  ideal of $R$ contained in $J$ and containing $y^m$. Also
$J\cap A=v^{(m)}A$. As $\mathcal{Z}$ is central, $IR_\mathcal{Z}$ is a non-zero ideal of the simple ring $R_\mathcal{Z}$ so, by \cite[Proposition 2.1.16(iv)]{McCR},  it follows that $p_m(-z,c_1,\dots,c_t)\in I$ for some  non-zero polynomial $p_m(X,X_1,\dots,X_t)\in \K[X,X_1,\dots,X_t]$.
 Thus $p_m(u-xy,c_1,\dots,c_t)\in J$ and, as $x\in J$ and $uxy=xyu\in
J$, it follows that $p_m(u,c_1,\dots,c_t)\in J\cap A=v^{(m)}A$.
\end{proof}

We next look at a class of algebras, introduced by Terwilliger and Worawannotai \cite{terwor}, to which Proposition~\ref{zlocsimple2} applies with $t=1$.

\begin{ex}
Let $A=\K[c,k^{\pm 1}]$, let $q\in \K^*$ and suppose that $q$ is not a root of unity. Let $\alpha$ be the $\K$-automorphism such that
$\alpha(k)=q^2 k$ and $\alpha(c)=c$. Fix a non-zero integer $n$ and a Laurent polynomial $f(k)=\sum a_ik^i\in \K[k,k^{-1}]$, such that $a_n=0$.
Let \[u=ck^n+f(k)\text{ and }v=u-\alpha(u)=(1-q^{2n})ck^n+\sum b_ik^i,\] where each $b_i=(1-q^{2i})a_i$. In particular $b_0=0$.
Then $R=R(A,\alpha,v,1)$ is generated by $k^{\pm 1}, c, x$ and $y$ subject to the relations
\begin{eqnarray}
ck&=&kc,\quad xc=cx,\quad yc=cy,\\
kk^{-1}&=&1=k^{-1}k,\\
xk&=&q^{-2}kx,\quad yk=q^{2}ky,\\
xy-yx&=&(1-q^{2n})ck^n+\sum b_ik^i.\label{xyyx}
\end{eqnarray}
By \eqref{xyyx},
\[
c=(1-q^{2n})^{-1}(xy-yx-\sum b_ik^i)k^{-n}
\]
so, as a generator, $c$ is redundant.
Substituting the above expression for $c$ in the relations $xc=cx$ and $cy=yc$ gives two relations in $x, y$ and $k$ that are cubic in $x, y$. Then $R$ is generated by $k^{\pm 1}, x$ and $y$ subject to these two relations and
\begin{eqnarray}
kk^{-1}&=&1=k^{-1}k,\\
xk&=&q^{-2}kx,\quad yk=q^{2}ky,\\
xy-yx&=&(1-q^{2n})ck^n+\sum b_ik^i.
\end{eqnarray}

This corresponds to the presentation in \cite[Definition 2.1]{terwor}, but the generators there are $e=q^{-t}k^{s}x$ and $f=y$, where $t-s=n$. Following \cite{terwor}, we shall refer to $R$ as an \emph{augmented down-up algebra}.

By the construction above, $R$ is conformal with central Casimir element
$z=xy-u$ and it is readily checked that $Z(R)=\K[c,z]$.  For $m\geq 1$,
\[v^{(m)}=(1-q^{mn})ck^n+\sum (1-q^{2im})a_ik^i\]
so $A/v^{(m)}A\simeq \K[k^{\pm 1}]$ which is an integral domain of transcendence degree $1$. Hence there exists a non-zero polynomial $p(X,Y)\in \K[X,Y]$ such that $p(u,c)\in v^{(m)}A$. Applying Proposition \ref{zlocsimple2}, we obtain the following result.
\end{ex}

\begin{prop}\label{centlocsimple}
If $R$ is an augmented down-up algebra then every non-zero ideal of $R$ has non-zero intersection with $Z(R)$ and the localization of $R$ at $Z(R)\backslash\{0\}$ is simple.
\end{prop}

\begin{cor}
An augmented down-up algebra $R$ is a UFD (in the sense of \cite{Chatt}).
\end{cor}
\begin{proof}
Certainly $R$ is a domain. It follows from Proposition~\ref{zlocsimple} that if $P$ is a height one prime ideal of $R$ then $f\in P$ for some irreducible element $f\in \K[c,z]$. It remains to show that $fR$ is completely prime. By \cite[Corollary 2.6]{inv}, $R$ is isomorphic to the generalized Weyl algebra $W=W(B,\alpha,u)$, where
$B=\K[c,k^{\pm 1},z]$, $\alpha(c)=c$, $\alpha(k)=q^2k$ and $\alpha(z)=z$.
Applying Lemma~\ref{GWAiso} below, with $I=fB$, we see that $R/fR$ is a generalized Weyl algebra over the domain $B/fB$ and hence is a domain.
\end{proof}

\begin{lemma}\label{GWAiso}
Let $W=W(A,\alpha,u)$ be a generalized Weyl algebra and let $I$ be an ideal of $A$ such that $I=\alpha(I)$. Then $IW$ is an ideal of $W$ and  $W/IW\simeq  W(A/I,\overline{\alpha},\overline{u})$, where $\overline{\alpha}$ is the automorphism of $A/I$ induced by $\alpha$ and $\overline{u}=u+I$.
\end{lemma}
\begin{proof}
It is routine to check that an isomorphism is given by
\[(a_i Y^i+\dots+a_0+\dots a_{-j}X^j)+IW\mapsto (\overline{a_i} Y^i+\dots \overline{a_0} + \overline{a_{-j}}X^j) ,\]
where, for $i\in \Z$, $\overline{a_i}=a_i+I$.
\end{proof}

\section{Families of exceptional simple factors}
Although  the results of this section are more widely applicable, they are aimed at the case where $R$ satisfies the hypotheses and the simplicity criterion  of Proposition~\ref{zlocsimple}. Examples include Examples \ref{usltwo}, \ref{usltwoq} and \ref{quantum torus}. We continue to assume that $\K$ is algebraically closed so that every height one prime ideal $P$ of $R$ has the form $(z-\lambda)R$ with $\lambda\in \K$. The factor $R/(z-\lambda)R$ is then the generalized Weyl algebra $W(A,\alpha,u+\lambda)$ and the following
result from \cite{vlad5} is applicable. An earlier version appeared in \cite{primitive}, where $A$ is commutative, and a more general version is \cite[Theorem 5.4]{ambiskew}.
\begin{theorem}\label{Wsimple} Let $\alpha$  be a $\K$-automorphism of an $\K$-algebra $A$, let $u\in A$ be central  and let $W$ be the generalized Weyl algebra $W(A,\alpha,u)$.
\label{t-thm} Then $W$ is simple if and only if
\begin{enumerate}
\item  $A$ is $\alpha$-simple;
\item   $\alpha^{m}$ is outer for all $m\geq 1$;
\item  $u$ is regular;
\item   $uA+\alpha^{m}(u)A=A$ for all $m\geq 1$.
\end{enumerate}
\end{theorem}

\begin{cor}\label{zlambdamax}
Let $R$ be a conformal ambiskew polynomial ring of the form $R(A,\alpha,u-\alpha(u),1)$, where $u\in Z(A)$ and the $\K$-algebra $A$ is a domain such that $A[y^{\pm 1},\alpha]$ is simple and $Z(A[y^{\pm 1};\alpha])=\K$. Let $\lambda\in \K$. The ideal $(z-\lambda)R$ is maximal if and only if $(u+\lambda)A+\alpha^m(u+\lambda)A=A$ for all $m\geq 1$.
\end{cor}
\begin{proof} Recall that $R/(z-\lambda)R\simeq W(A,\alpha,u+\lambda)$. From Remark~\ref{altcon}, we know that, as $A[y^{\pm 1},\alpha]$ is simple and $Z(A[y^{\pm 1};\alpha])=\K$, $A$ is $\alpha$-simple and $\alpha^m$ is outer for all $m\geq 1$.
If $u+\lambda=0$ then $W(A,\alpha,u+\lambda)$ is not simple, by Theorem~\ref{Wsimple}, and $(u+\lambda)A+\alpha^m(u+\lambda)A=0\neq A$ so we can assume that $u+\lambda\neq 0$
in the domain $A$.
Thus Conditions (i), (ii) and (iii) in Theorem~\ref{Wsimple} hold for $W(A,\alpha,u+\lambda)$ and $(z-\lambda)R$ is maximal if and only if $(u+\lambda)A+\alpha^m(u+\lambda)A=A$ for all $m\geq 1$.
\end{proof}

\begin{ex}\label{usltwo2}
Let $R$ be as in Example~\ref{usltwo}. Thus $\ch(\K)=0$, $A=\K[t]$, $\alpha(t)=t+2$, $\rho=1$, $u=\frac{-1}{4}(t-1)^2$, $v=t$ and $R$ is the enveloping algebra
$U(sl_2)$. Then every height one prime ideal of $R$ has the form $(z-\lambda)R$ for some $\lambda\in \K$ and $R/(z-\lambda)R=W(\K[t],\alpha,u+\lambda)$. For $m\geq 1$, let $M_{m,\lambda}=(u+\lambda)A+\alpha^m(u+\lambda)A$ which, as $v^{(m)}=u-\alpha^m(u)=u+\lambda-\alpha^m(u+\lambda)$, is equal to $(u+\lambda)A+v^{(m)}A$.
We have seen in \ref{usltwo} that, for $m\geq 1$, $v^{(m)}=m(t+m-1)$ so $v^{(m)}A=(t-(1-m))A$. Also \[u+\lambda\equiv \left(\lambda-\frac{1}{4}m^2\right)\bmod v^{(m)}A.\]
If $\lambda\neq \frac{1}{4}m^2$ for all $m\in \N$ then $M_{m,\lambda}=A$ and, by Corollary~\ref{zlambdamax}, $(z-\lambda)R$ is maximal. On the other hand, if $\lambda=\frac{1}{4}m^2$ for some, necessarily unique, $m\in \N$ then $M_{m,\lambda}=v^{(m)}A=(t+m-1)A$ is maximal and, by Corollary~\ref{zlambdamax},   $(z-\lambda)$ is not maximal.
\end{ex}

\begin{ex}\label{usltwoq2} Let $R$ be the quantum enveloping algebra
$U_q(sl_2)$  as in Example~\ref{usltwoq}. Thus $q\in \K^*$ is not a root of unity, $A=\K[t^{\pm 1}]$,  $\alpha(t)=q^2t$, $\rho=1$, and $u=-(q^{-1}t+qt^{-1})/(q-q^{-1})^2$.
Every height one prime ideal of $R$ has the form $(z-\lambda)R=(xy-(u+\lambda))R$ for some $\lambda\in \K$ and $R/(z-\lambda)R=W(\K[t^{\pm 1}],\alpha,u+\lambda)$, where $\alpha(t)=q^2t$. We have seen in \ref{usltwoq2} that, for $m\geq 1$,
\[
v^{(m)}=\frac{q^{2m-1}-q^{-1}}{(q-q^{-1})^2}(t-q^{2-2m}t^{-1}).
\]
For $m\geq 1$, let \[M_{m,\lambda}=(u+\lambda)A+\alpha^m(u+\lambda)A=(u+\lambda)A+v^{(m)}A.\]
Then  $t^{-1}\equiv q^{2m-2}t\bmod{(v^{(m)}A)}$ from which it follows that $M_{m,\lambda}$ contains the  ideal $(t^2-q^{2-2m})A$ and the maximal ideal
$(t-\mu)A$ where \[\mu=\frac{\lambda(q-q^{-1})^2}{q^{-1}+q^{2m-1}}.\] It now follows that
\[M_{m,\lambda}\neq A\Leftrightarrow M_{m,\lambda}\text{ is maximal }\Leftrightarrow\mu^2=q^{2-2m}\Leftrightarrow \mu=\pm q^{1-m}\Leftrightarrow\lambda=\pm\frac{q^{-m}+q^m}{(q-q^{-1})^2}.\]
By Corollary~\ref{zlambdamax} the ideal $(z-\lambda)R$ is maximal if and only if $\lambda \neq \pm\frac{q^{-m}+q^m}{(q-q^{-1})^2}$ for all $m\in \N$.
\end{ex}

The following lemma determines those values of $\lambda$ for which $R/(z-\lambda)R$ is simple in Example~\ref{quantum torus}.
\begin{lemma}\label{countableex} Suppose that $q$ is not a root of unity.
Let $A$, $u=q^{\frac{p-1}{2}}z_{p}^{-1}+qz_{p}$ and  $\alpha$ be as in Example~\ref{quantum torus} and let $\lambda\in \K$. Let $m\in \N$. Then the ideal $(u+\lambda)A+\alpha^m(u+\lambda)A$ is proper if and only if $\lambda=\pm q^\frac{p-2m+1}{4}(q^m+1)$. If $\lambda=\pm q^\frac{p-2m+1}{4}(q^m+1)$ then
$(u+\lambda)A+\alpha^m(u+\lambda)A$ is a maximal (and completely prime) ideal of $A$ and $(u+\lambda)A+\alpha^a(u+\lambda)A=A$
for all $a\in \N\backslash\{m\}$.
\end{lemma}

\begin{proof}
Suppose that $(u+\lambda)A+\alpha^a(u+\lambda)A$ is proper. Let $B$ be the subalgebra of $A$ generated by $z_1^{\pm1}, z_2^{\pm1},\dots,z_{p-1}^{\pm1}$. As for $S$ in Example~\ref{quantum torus}, it follows from
\cite[Proposition 1.3]{McCP} that $B$ is simple and $Z(B)=\K$.  It then follows from \cite[Lemma 9.6.9(i)]{McCR} that
the maximal ideals of $A$ have the form $(z_{p}-\mu)A$, $\mu\in \K^*$, and are completely prime with factors isomorphic to quantum tori in $p-1$ indeterminates. So there exists $\mu\in \K^*$ such that $u+\lambda\in (z_{p}-\mu)A$ and $\alpha^m(u+\lambda)\in (z_{p}-\mu)A$ and hence such that
\[
q^\frac{p-1}{2}\mu^{-1}+\lambda+q\mu=0=q^mq^\frac{p-1}{2}\mu^{-1}+\lambda+q^{1-m}\mu.
\]
Eliminating the terms that involve $\mu^{-1}$,
\[\lambda(q^m-1)+(q^{m+1}-q^{1-m})\mu=0\]
and, dividing through by $q^m-1$, which is necessarily non-zero, $\lambda=-q^{1-m}(q^m+1)\mu$. Hence $\lambda\neq 0$. Also
\begin{eqnarray*}
0&=&q^\frac{p-1}{2}q^{1-m}(q^m+1)\lambda^{-1}-\lambda+q^m(q^m+1)^{-1}\lambda,\\
0&=&q^\frac{p-1}{2}q^{1-m}(q^m+1)^2-\lambda^2(q^m+1)+q^m \lambda^2,\\
0&=&q^\frac{p-1}{2}q^{1-m}(q^m+1)^2-\lambda^2\text{ and}\\
\lambda&=&\pm q^\frac{p-2m+1}{4}(q^m+1).
\end{eqnarray*}

Conversely,  suppose that $\lambda=\pm q^\frac{p-2m+1}{4}(q^m+1)$. Then
\begin{eqnarray*}
u+\lambda&=&(z_{p}\pm q^\frac{p+2m-3}{4})(q\pm  q^\frac{p-2m+1}{4}z_p^{-1})  \text{ and}\\
\alpha^m(u+\lambda)&=&(z_{p}\pm  q^\frac{p+2m-3}{4})q^{1-m}(1\pm q^\frac{p+6m-3}{4}z_{p}^{-1}).
\end{eqnarray*}
Thus \[(u+\lambda)A+\alpha^m(u+\lambda)A\subseteq (z_{p}\pm q^\frac{p+2m-3}{4})A\neq A.\] Moreover, as $q\pm  q^\frac{p-2m+1}{4}z_{p}^{-1}$ and $1\pm q^\frac{p+6m-3}{4}z_{p}^{-1}$ generate distinct maximal ideals,
\[(u+\lambda)A+\alpha^m(u+\lambda)A=(z_{p}\pm q^\frac{p+2m-3}{4})A,\] which is a maximal  (and completely prime) ideal of $A$.

Finally, if $(u+\lambda)A+\alpha^a(u+\lambda)A\neq A$ then
\[q^\frac{p-2a+1}{4}(q^a+1)=\pm \lambda=\pm q^\frac{p-2m+1}{4}(q^m+1)\]
 from which it follows successively  that
\begin{eqnarray*}
q^{\frac{-a}{2}}(q^a+1)&=&\pm q^{\frac{-m}{2}}(q^m+1),\\
q^{\frac{a}{2}}+q^{\frac{-a}{2}}&=&\pm(q^{\frac{m}{2}}+q^{\frac{-m}{2}}),\\
q^{a}+q^{-a}&=&(q^m+q^{-m})\text{ and}\\
q^a-q^m&=&(q^a-q^m)q^{-a-m}.\end{eqnarray*} As $q$ is not a root of unity, this cannot happen if $a\in \N\backslash\{m\}$.
\end{proof}

\begin{cor}\label{whenmax}
If $R$ is as in Example~\ref{quantum torus},  then $(z-\lambda)R$ is maximal if and only if, for all $m\geq 1$, $\lambda\neq \pm q^\frac{p-2m+1}{4}(q^m+1)$.
\end{cor}
\begin{proof}
This is immediate from Corollary~\ref{zlambdamax} and Lemma~\ref{countableex}.
\end{proof}

We now aim to establish conditions that, in the context of Corollary~\ref{zlambdamax}, will imply that when $(z-\lambda)R$ is not maximal there is a unique non-zero prime ideal in $R/(z-\lambda)R$.

\begin{lemma}\label{YmorXm}
Let $W=W(A,\alpha,u)$ be a generalized Weyl algebra with $u$ central in $A$. Let $j\geq 1$
 be such that $uA+\alpha^j(u)A=A$. Let $J$ be an ideal of $W$. If $Y^j\in J$ then  $Y^{j-1}\in J$ and if  $X^j\in J$ then $X^{j-1}\in J$. Consequently,
 if  $uA+\alpha^i(u)A=A$ for $1\leq i\leq j$ and $Y^j\in J$ or $X^j\in J$ then $J=W$.
\end{lemma}

\begin{proof}
If $Y^j\in J$ then \[uY^{j-1}=XY^j\in J\text{ and }\alpha^j(u)Y^{j-1}=Y^{j-1}\alpha(u)=Y^jX\in J,\] whence \[AY^{j-1}=(Au+A\alpha^j(u))Y^{j-1}\subseteq J\] and so $Y^{j-1}\in J$. Similarly, if $X^j\in J$ then \[\alpha(u)X^{j-1}=YX^j\in J\text{ and }\alpha^{-(j-1)}(u)X^{j-1}=X^{j-1}u=X^jY\in J,\] whence \[AX^{j-1}=(A\alpha^{-(j-1)}(u)+A\alpha(u))X^{j-1}\subseteq J\] and so $X^{j-1}\in J$. Repeating the argument yields the stated consequence.
\end{proof}

\begin{prop}\label{J}
Let $W=W(A,\alpha,u)$ be a generalized Weyl algebra with $u$ central in $A$. Let $m\geq 1$
 be such that $uA+\alpha^j(u)A=A$ for $1\leq j<m$ but $uA+\alpha^m(u)A\neq A$. Let $I$ be an ideal of $A$ containing $uA+\alpha^m(u)A$. There is a
  $\Z$-graded ideal $J=J(I)$ of $W$ such that,
  for $i\geq 0$, $J_i=I_iY^i$ and $J_{-i}=I_{-i}X^i$, where, if $i\geq m$ then $I_i=I_{-i}=A$ and, if $0\leq i\leq m-1$ then
\[I_{i}:=\cap_{\ell=0}^{m-1-i} \alpha^{-\ell}(I)\text{ and }I_{-i}:=\cap_{\ell=i}^{m-1} \alpha^{-\ell}(I).\]
\end{prop}
\begin{proof} Note that the two definitions of $I_0$ coincide.
With $J_i$ as above for $i\in \Z$, let $J=\oplus_{i\in \Z} J_i$. It is clear that $J_iA\subseteq J_i$ and $AJ_i\subseteq J_i$ for each $i\in \Z$.
Let $i\geq 0$. Clearly $J_iY\subseteq J_{i+1}$ and $J_{-i}X\subseteq J_{-(i+1)}$.
Also,
\[YJ_i\subseteq \alpha(I_i)Y^{i+1}\subseteq I_{i+1}Y^{i+1}=J_{i +1}\]
and, similarly, $XJ_{-i}\subseteq J_{-(i+1)}$.
Now let $i\geq 1$. As $u\in \alpha^{-m}(I)$ and $u\in I$,
\[
J_iX=I_iY^{i-1}\alpha(u)=I_i\alpha^i(u)Y^{i-1}\subseteq I_i\alpha^{i-m}(I)Y^{i-1}\subseteq I_{i-1}Y^{i-1}=J_{i-1}\] and
\[
XJ_i=\alpha^{-1}(I_i)XY^i=\alpha^{-1}(I_i)uY^{i-1}\subseteq \alpha^{-1}(I_i)IY^{i-1}\subseteq  I_{i-1}Y^{i-1}=J_{i-1}.\]
Similarly,
 $J_{-i}Y\subseteq J_{-(i-1)}$ and $YJ_{-i}\subseteq J_{-(i-1)}$. It follows that $J$ is a graded ideal of $W$.
\end{proof}

\begin{notn}\label{de}
For $i\geq 1$, let $d_i=\alpha(u)\alpha^2(u)\dots \alpha^i(u)$ and $e_i=\alpha^{-i}(d_i)=u\alpha^{-1}(u)\dots \alpha^{1-i}(u)$. Thus
$d_i=Y^iX^i$ and $e_i=X^iY^i$, see Definitions~\ref{GWA}.
\end{notn}

\begin{lemma}\label{YXAI}
Let $W$ and $m$ be as in Proposition~\ref{J}. For $0<i<m$,
$d_iA+uA=A=d_iA+\alpha^m(u)A$ and $e_iA+\alpha^{-i}(u)A=A=e_iA+\alpha^{m-i}(u)A$.
\end{lemma}
\begin{proof}
Suppose that $d_iA+uA\neq A$ and let $M$ be a maximal ideal of $A$ containing $d_iA+uA$. As $u$ is central, there exists $j$ such that $1\leq j\leq i<m$,
 $\alpha^j(u)\in M$ and $u\in M$. This contradicts the conditions of Proposition~\ref{J} so $d_iA+uA=A$. Similarly $d_iA+\alpha^m(u)A=A$ and, applying $\alpha^{-i}$, $e_iA+\alpha^{-i}(u)A=A=e_iA+\alpha^{m-i}(u)A$.
\end{proof}

\begin{lemma}\label{IJmax}
Let $W$, $I$ and $J=J(I)$ be as in Proposition~\ref{J} and suppose that $I$ is a maximal ideal of $A$. Then $J$ is a maximal ideal of $W$.
\end{lemma}
\begin{proof} Recall that $X^m\in J$ and $Y^m\in J$.
Let $M$ be an ideal of $W$ such that $J\subset M$. There exist $a_{-(m-1)},\dots,a_0,\dots,a_{m-1}\in A$ such that
\[g:=a_{-(m-1)}X^{m-1}+\dots+a_0+\dots+a_{m-1}Y^{m-1}\in M\] and $a_i\notin I_i$ for at least one $i$ with $m-1\geq i\geq -(m-1)$.
Suppose that $a_i\notin I_i$ for some  $i$ with $0\leq i\leq m-1$. Then there exists $\ell$ such that $0\leq \ell\leq i\leq m-1$ and
$\alpha^\ell(a_i)\notin I$. But $Y^\ell gY^{m-1-i-\ell}\in M$ and its coefficient of $Y^{m-1}$ is $\alpha^\ell(a_{i})$. Replacing $g$ by $Y^\ell gY^{m-1-i-\ell}$ and recalling that $Y^m\in J\subseteq M$, we can assume that $i=m-1$. Thus $a_{m-1}\notin I_{m-1}=I$. Let $F$ denote the set of all elements $f\in A$ for which there exist $b_{-(m-1)},\dots,b_{m-2}\in A$ such that
 \[b_{-(m-1)}X^{m-1}+\dots+b_0+\dots+b_{m-2}Y^{m-2}+fY^{m-1}\in M.\] Then $F$
is an ideal of $A$, $a_{m-1}\in F\backslash I$ and $I\subseteq F$ so $F\neq I$ and, by the maximality of $I$, $F=A$. Hence we may assume that $a_{m-1}=1$ and that
\[g=w_{1-m}+\dots+w_0+\dots+w_{m-2}+Y^{m-1}\in M\] where, for $-(m-1)\leq i\leq m-2$, $w_i$ is homogeneous of degree $i$.
Consider \[X^{m-1}gY^{m-1}=X^{m-1}w_{1-m}Y^{m-1}+\dots+X^{m-1}w_0Y^{m-1}+\dots+X^{m-1}Y^{m-1}Y^{m-1}\in M.\] The term $X^{m-1}Y^{m-1}X^{m-1}$ is homogeneous of degree $1-m$ and the other terms have degree $\leq -m$. As $W_{-k}=AX^k$ for $k\geq m$ and $X^m\in J\subset M$, the other terms  are in $M$.
Hence $X^{m-1}Y^iX^{m-1}\in M$, that is, $e_{m-1}X^{m-1}\in M$,  where,  as in Notation~\ref{de}, $e_{m-1}=u\alpha^{-1}(u)\dots\alpha^{-(m-2)}(u)$.
As $I_{1-m}=\alpha^{-(m-1)}(u)A$, $\alpha^{-(m-1)}(u)X^{m-1}\in M$. Hence $(\alpha^{-(m-1)}(u)A+e_{m-1}A)X^{m-1}\subset M$ and it follows from Lemma~\ref{YXAI} that
$X^{m-1}\in M$.
By Lemma~\ref{YmorXm}, $M=W$.

The argument if $a_i\notin I_i$ for some  $i$ with $0> i\geq 1-m$ is similar. We may assume that $i=1-m$ and $a_{-(1-m)}=1$ and consider $Y^{m-1}gY^{m-1}$, which belongs to $M$, giving that
 $Y^{m-1}X^{m-1}Y^{m-1}=d_{m-1}Y^{m-1}\in M$, which leads us to conclude, using Lemma~\ref{YXAI} and the fact that $uX^{m-1}\in J\subseteq M$, that  $M=W$.
 This completes the proof that $J$ is maximal.
\end{proof}

\begin{lemma}\label{XnYn}
Let $W$ be as in Proposition~\ref{J}, let $I=uA+\alpha^m(u)A$ and let $J=J(I)$ be as in Proposition~\ref{J}. Any prime ideal $P$ of $W$ containing $X^m$ and $Y^m$ must contain $J$.
\end{lemma}
\begin{proof}
For $i\geq 0$, let $d_i=Y^iX^i=\alpha(u)\alpha^2(u)\dots\alpha^i(u)$.
Let $K$ be an ideal of $W$ that contains $X^m$ and $Y^m$.

We claim that  $d_{m-1}J\subseteq K$.
For this it suffices to show that $d_{m-1}J_i\subseteq K$ for all $i\in \Z$. As $Y^m\in K$ and  $X^m\in K$, $J_{i}\subseteq K$ and $J_{-i}\subseteq K$ for $i\geq m$. Let $1\leq i<m$.
Then $J_{-i}\subseteq X^iA$ so\[d_{m-1}J_{-i}\subseteq d_{m-1}X^{i}A=Y^{m-1}X^{m-1}X^{i}A=Y^{m-1}X^mX^{i-1}A\subseteq K.\]

Also, for $0\leq i\leq m$, $J_i\subseteq IY^i$ so
\[d_{m-1}J_i\subseteq d_{m-1}IY^i=ud_{m-1}Y^iA+\alpha^m(u)d_{m-1}Y^iA=ud_{m-1}Y^iA+d_mY^iA.\]
Here $d_m=Y^mX^m\in K$ and $ud_{m-1}=XYY^{m-1}X^{m-1}=XY^mX^{m-1}\in K$ so $d_{m-1}J_i\subseteq K$. This completes the proof of the
claim that  $d_{m-1}J\subseteq K$.

Now suppose that $K$ is prime and that
$J\not\subseteq K$. Then, as $J$ is an ideal and  $d_{m-1}J\subseteq K$,  $d_{m-1}\in K$.  Note that $X^{m-1}u=X^{m-1}XY=X^mY\in K$ so
 $X^{m-1}(uA+d_{m-1}A)\subseteq K$.
It follows, by Lemma~\ref{YXAI}, that $X^{m-1}\in K$.  By Lemma~\ref{YmorXm},
$K=W$. This is a contradiction so $J\subseteq K$.
\end{proof}

\begin{theorem}\label{YmorXmthm}
Let $W(A,\alpha,u)$ be a generalized Weyl algebra, with $u$ central and regular in $A$, such that, for some fixed $m \in \mathbb{N}$:
\begin{enumerate}
\item $Au+A\alpha^i(u)=A$ for all $i\in \N\backslash\{m\}$;
\item $M:=Au+A\alpha^m(u)$ is a maximal ideal in $A$.
\end{enumerate}
Then the ideal $J(M)$ is a maximal ideal of $W$ containing both $X^m$ and $Y^m$ and is the unique prime ideal $P$ in $W$ for which there exists $r \in \mathbb{N}$ such that $X^r\in P$ and $Y^r\in P$. Moreover if $A$ is $\alpha$-simple and no power of $\alpha$ is inner then $J(M)$ is the unique non-zero
prime ideal in $W$.
\end{theorem}

\begin{proof}
 By Lemmas \ref{IJmax} and \ref{XnYn} respectively, $J(M)$ is maximal and is the unique prime ideal in $W$ containing $X^m$ and $Y^m$.

Let $K$ be an ideal of $W$ containing $X^r$ and $Y^r$ for some $r \in \mathbb{N}$. By Lemma \ref{YmorXm}, if $0<r<m$ then $K=W$ and if  $r>m$ then $X^m\in K$ and $Y^m\in K$. Hence $J(M)$ is the unique prime ideal $P$ in $W$ for which there exists $r \in \mathbb{N}$ such that $X^r\in P$ and $Y^r\in P$.

Now suppose that $A$ is $\alpha$-simple and that no power of $\alpha$ is inner.
Let $P$ be a non-zero prime ideal of $W$.
Recall from Definitions~\ref{GWA} that
$A[Y^{\pm 1};\alpha]$ and $A[X^{\pm 1};\alpha^{-1}]$ are the localizations of $W$ at the Ore sets $\{Y^i: i\geq 1\}$ and $\{X^i: i\geq 1\}$  respectively. These rings are simple, by \cite[Theorem 1.8.5]{McCR}, so, by \cite[Lemma 3.1]{ambiskew}, there exist $r,s$ such that $X^r\in P$ and $Y^s\in P$. Replacing $r$ and $s$ by their maximum, we can assume that $r=s$. By the above $P=J(M)$.
\end{proof}

\begin{cor}\label{unique}
Let $R$ be a conformal ambiskew polynomial ring of the form $R(A,\alpha,u-\alpha(u),1)$, where $u\in Z(A)\backslash \K$ and the $\K$-algebra $A$ is a domain such that $A[y^{\pm 1},\alpha]$ is simple and $Z(A[y^{\pm 1};\alpha])=\K$. Let $\lambda\in \K$ be such that
$(u+\lambda)A+\alpha^m(u+\lambda)A\neq A$ for some $m\geq 1$. If the ideal $(u+\lambda)A+\alpha^m(u+\lambda)A$ is maximal and $(u+\lambda)A+\alpha^n(u+\lambda)A=A$ for all $n\in \N\backslash\{m\}$  then $R/(z-\lambda)R$ has a unique non-zero prime ideal.
\end{cor}
\begin{proof}
This is immediate from Theorem~\ref{YmorXmthm} using the isomorphism between $R/(z-\lambda)R$ and $W(A,\alpha,u+\lambda)$.
\end{proof}

In the case of $U(sl_2)$ and $U_q(sl_2)$, the maximal ideals that arise in the form $J(M)$ are the annihilators of the finite-dimensional simple modules. These are well understood and provide nice illustrations of the theory developed above.
\begin{ex}\label{usltwo3}
Let $R$ be the enveloping algebra
$U(sl_2)$ as in Example~\ref{usltwo} and Example~\ref{usltwo2}. Thus $\ch(\K)=0$, $A=\K[t]$, $\alpha(t)=t+2$, $\rho=1$,
$u=\frac{-1}{4}(t-1)^2$ and $v=t$. We have seen that each $v^{(m)}=m(t+m-1)$, whence $v^{(m)}A$ is maximal in $A$, and that if $(z-\lambda)R$ is not maximal in $R$ then $\lambda=\frac{1}{4}m^2$ for some $m\in \N$. In this case,
 \[(u+\lambda)A+\alpha^m(u+\lambda)A=v^{(m)}A\] and
\[(u+\lambda)A+\alpha^n(u+\lambda)A=A\text{ for all }n\in \N\backslash\{m\}.\]
By Corollary~\ref{unique},   $R/(z-\lambda)R$ has a unique non-zero prime ideal.
\end{ex}

\begin{ex}\label{usltwoq3} Let $R$ be the quantum enveloping algebra
$U_q(sl_2)$  as in Example~\ref{usltwoq} and Example~\ref{usltwoq2}. Thus $q\in \K^*$ is not a root of unity, $A=\K[t^{\pm 1}]$,  $\alpha(t)=q^2t$, $\rho=1$, and $u=-(q^{-1}t+qt^{-1})/(q-q^{-1})^2$.
Every height one prime ideal of $R$ has the form $(z-\lambda)R$ and $R/(z-\lambda)R$ can be identified with $W(\K[t^{\pm 1}],\alpha,u+\lambda)$. Let $\lambda\in \K$, $m\in \N$ and $M_{m,\lambda}=(u+\lambda)A+\alpha^m(u+\lambda)A$. We have seen in Example~\ref{usltwoq2} that $M_{m,\lambda}\neq A$ if and only if is maximal if and only if $\lambda=\pm\frac{q^{-m}+q^m}{(q-q^{-1})^2}$.
Let $\lambda=\pm\frac{q^{-m}+q^m}{(q-q^{-1})^2}$. For $n\in\N\backslash\{m\}$,
 \[q^{-n}+q^n=\pm(q^{-m}+q^m)\Rightarrow (q^m\mp q^n)(1\mp q^{-(m+n)})=0,\]
so, as $q$ is not a root of unity, $M_{n,\lambda}=A$. It now follows from Corollary~\ref{unique} that $R/(z-\lambda)R$ has a unique non-zero prime ideal $J(M_{m,\lambda})$.
\end{ex}

In the next example, the exceptional maximal ideals $J(M)$ have infinite codimension over $\K$ and so are not annihilators of finite-dimensional simple modules.

\begin{ex}\label{spectrumSC}
Let $p\geq 1$ be odd, let $q\in \K^*$ and suppose that $q$ is not a root of unity. Let $R=R(A,\alpha,v,1)$ and its Casimir element $z$ be as in Example~\ref{quantum torus}.
We have seen that the height one prime ideals of $R$ are the ideals $(z-\lambda)R$ and, in Corollary~\ref{whenmax} that $(z -\lambda)R$ is maximal unless
\[\lambda=\pm q^\frac{p-2m+1}{4}(q^m+1)\text{ for some }m\in \N.\] To complete the analysis of the spectrum of $R$, let $m\in \N$ and let
\[\lambda=\pm q^\frac{p-2m+1}{4}(q^m+1).\]
It follows from Theorem~\ref{YmorXmthm},  together with Lemma~\ref{countableex} and its proof, that in this case
 $R/(z-\lambda)R$, which we identify with $W(A,\alpha,u+\lambda)$, has a unique non-zero prime ideal $J((z_{p}\pm q^\frac{p+2m-3}{4})A)$.
 Therefore the prime spectrum of $R$ consists of $0$, the height one prime ideals $(z-\lambda)R$, $\lambda\in \K$, and countably many height two prime ideals
\[F_{m,1}=\pi^{-1}(J((z_{p}-q^\frac{p+2m-3}{4})A))\text{ and }F_{m,-1}=\pi^{-1}(J((z_{p}+q^\frac{p+2m-3}{4})A)),\] where $m\in \N$ and each
$\pi:R\rightarrow R/(z-\lambda)R$ is the appropriate canonical epimorphism.
\end{ex}

\section{Goldie rank}
In Examples \ref{usltwo2}, \ref{usltwoq2} and \ref{spectrumSC}, the height one prime ideals are principal, generated by translates of the Casimir element, all but countably many of these are maximal and the other maximal ideals have height two. For $U(sl_2)$ in \ref{usltwo2} and $U_q(sl_2)$ in \ref{usltwoq2}, the height two maximals are annihilators of finite-dimensional simple modules and so the factor rings are matrix rings over $\K$. For $U(sl_2)$, there is one simple module of each dimension $d\in \N$ and so there is a unique height two maximal ideal of Goldie rank $d$. For $U_q(sl_2)$, there are two height two maximal ideals of Goldie rank $d$. In Example \ref{spectrumSC}, the simple factor rings $R/F_{m,1}$ and $R/F_{m,-1}$ are infinite-dimensional and hence not isomorphic to matrix rings over $\K$. It is the purpose of this section to show that, in the situation
of  Theorem~\ref{YmorXmthm}, but with the further condition that $A/M$ is a right Ore domain, the factor  $W/J(M)$ has
  Goldie rank $m$.
	\begin{notn}\label{W}For the remainder of the paper, let
$W=W(A,\alpha,u)$ be a generalized Weyl algebra, with $u$ central and regular in $A$, such that, for some fixed $m \in \mathbb{N}$,
$M:=Au+A\alpha^m(u)$ is such that $A/M$ is a simple right Ore domain and
 $Au+A\alpha^i(u)=A$ for $i\in\N\backslash\{m\}$.\end{notn}

\begin{notn}\label{Mi}
In the notation of \ref{W} and for $0\leq i\leq m-1$, let $M_i=\alpha^{-i}(M)=A\alpha^{-i}(u)+A\alpha^{m-i}(u)=\alpha^{-i}(u)A+\alpha^{m-i}(u)A$. Thus each $M_i$ is a maximal ideal.
As the generators $\alpha^s(u)$ are central, $M_iM_j=M_jM_i$ for $0\leq i,j\leq m-1$. Also $M_0,M_1,\dots,M_{m-1}$ are distinct for if $0\leq i<j<m$ and $\alpha^{-i}(M)=M_i=M_j=\alpha^{-j}(M)$ then $\alpha^{j-i}(u)\in \alpha^{j-i}(M)=M$, which is impossible as $Au+A\alpha^{j-i}(u)=A$. So the following result applies.
\end{notn}

\begin{lemma}\label{sumsandcaps}
Let $R$ be a ring with $m$ commuting distinct maximal ideals $M_0,M_1,\dots,M_{m-1}$. Let $0\leq i_1,\dots,i_r,j_1,\dots,j_s,k_1,\dots,k_t<m$ be distinct integers.

\rm{(i)} $M_{i_1}\dots M_{i_r}+M_{j_1}\dots M_{j_s}=R$.

\rm{(ii)} $M_{k_1}\dots M_{k_t}M_{i_1}\dots M_{i_r}+M_{k_1}\dots M_{k_t}M_{j_1}\dots M_{j_s}=M_{k_1}\dots M_{k_t}.$

\rm{(iii)} For $0\leq i\leq m-1$, $M_0\cap M_1\cap \dots\cap M_{i}=M_0M_1\dots M_{i}$.

\end{lemma}
\begin{proof}
\rm{(i)} Suppose not. Then there exists a maximal ideal $M$  such that $M_{i_1}\dots M_{i_r}+M_{j_1}\dots M_{j_s}\subseteq M$. As $M$ is prime,
there exist $1\leq a\leq r$ and $1\leq b\leq s$ such that $M_{i_a}\subseteq M$ and $M_{j_b}\subseteq M$. But then, by maximality, $M_{i_a}=M=M_{j_b}$,
contrary to the hypotheses.

\rm{(ii)} This is immediate from (i) and the law $I(J+K)=IJ+IK$.

\rm{(iii)} We proceed by induction on $i$. The statement is certainly true when $i=0$ so we may assume that $i>0$ and that
$M_0\dots M_{i-1}=M_0\cap \dots\cap M_{i-1}$. Let $J=M_0\cap \dots\cap M_{i-1}=M_0\dots M_{i-1}$.
Then $J+M_{i}=R$, by (i), so, as the $M_j$'s commute,
\[J\cap M_i=(J\cap M_{i})(J+M_{i})\subseteq JM_{i}+M_{i}J=JM_{i}\subseteq J\cap M_{i},\]
whence $M_0\cap M_1\cap \dots\cap M_{i}=M_0M_1,\dots M_{i}$.
\end{proof}

Our aim now is to find $m$ uniform right ideals of the $\Z$-graded ring $W/J(M)$  whose sum is direct and equal to $W/J(M)$.

\begin{notn}\label{Pi} In the notation of \ref{W} and \ref{Mi} and for $0\leq i<j\leq m-1$ and $i\leq r\leq j$, we shall denote the product $M_iM_{i+1}\dots M_{r-1}M_{r+1}\dots M_{j}$ by
$\Pi(M,i,\widehat{r},j)$ and the product $M_iM_{i+1}\dots M_{j}$ by
$\Pi(M,i,j)$.

The components  $(W/J(M))_d$ and
$(W/J(M))_{-d}$ are $0$ if $d\geq m$.  If $0\leq d<m$ then, by Lemma~\ref{sumsandcaps}(iii),
\[(W/J(M))_d=AY^d/\Pi(M,0,m-1-d)Y^d\text{ and }(W/J(M))_{-d}=AX^d/\Pi(M,d,m-1)X^d.\]
Each $(W/J(M))_d$ is an $A-A$-bimodule while,  in accordance with the proof of Proposition~\ref{J}, right and left multiplication by $Y$, respectively $X$, give well-defined maps $(W/J(M))_d\rightarrow (W/J(M))_{d+1}$,  respectively
 $(W/J(M))_d\rightarrow (W/J(M))_{d-1}$.
\end{notn}

\begin{notn}\label{Jr} In the notation of \ref{W}, \ref{Mi} and \ref{Pi} and for $0\leq r\leq m-1$, let  $J^{(r)}$ be the graded right ideal \[(\Pi(M,0,\widehat{r},m-1)W+J(M))/J(M)\]
of  $W/J(M)$.
The $0$-component of $J^{(r)}$ is \[\Pi(M,0,\widehat{r},m-1)/\Pi(M,0,m-1)).\]
If $d>m-r-1$ then $J^{(r)}_d=0$ and if
$1\leq d\leq m-r-1$ then
\[J^{(r)}_d=\Pi(M,0,\widehat{r},m-1-d)Y^d/\Pi(M,0,m-1-d)Y^d.\]
If $r<d$ then
 $J^{(r)}_{-d}=0$ and if $1<d\leq r$, then
\[J^{(r)}_{-d}=\Pi(M,d,\widehat{r},m-1)X^d/\Pi(M,d,m-1)X^d.\]
\end{notn}

\begin{lemma}\label{Jsum} In the notation of \ref{Jr}, the sum $J^{(0)}+J^{(1)}+\dots+J^{(m-1)}$ is direct and equal to $W/J(M)$.
\end{lemma}

\begin{proof}
It suffices to show that, for $-m<d<m$, $J^{(0)}_d+J^{(1)}_d+\dots +J^{(m-1)}_d=W/J(M)_d$ and that the sum is direct.

Let $0<s<m$.
Then, by repeated use of Lemma~\ref{sumsandcaps}(ii),
\begin{eqnarray*}&&J^{(0)}_0+J^{(1)}_0+\dots+J^{(s)}_0\\
&=&(\Sigma_{j=0}^s\Pi(M,0,\widehat{j},m-1))/\Pi(M,0,m-1)\\
&=&\Pi(M,s+1,m-1)/\Pi(M,0,m-1).
\end{eqnarray*}

Similar calculations show that if $0<d<m$ then
\begin{eqnarray*}&&J^{(0)}_d+\dots+J^{(s)}_d\\
&=&(\Sigma_{j=0}^s\Pi(M,0,\widehat{j},m-d-1))Y^d/\Pi(M,0,m-1)Y^d\\
&=&\Pi(M,s+1,m-d-1)Y^d/\Pi(M,0,m-d-1)Y^d
\end{eqnarray*}
and
\begin{eqnarray*}&&J^{(0)}_{-d}+\dots+J^{(s)}_{-d}\\
&=&(\Sigma_{j=d}^s\Pi(M,d,\widehat{j},m-1))X^d/\Pi(M,d,m-1)X^d\\
&=&\Pi(M,s+1,m-1)X^d/\Pi(M,d,m-1)X^d.
\end{eqnarray*}
Taking $s=m-1$ above, \[J^{(0)}_0+J^{(1)}_0+\dots+J^{(m-1)}_0=
A/\Pi(M,0,m-1)=(W/J(M))_0\] and, for $0<d<m$, \[J^{(0)}_d+J^{(1)}_d+\dots+J^{(m-1)}_d=
AY^d/\Pi(M,0,m-d-1)Y^d=(W/J(M))_d\] and \[J^{(0)}_{-d}+J^{(1)}_{-d}+\dots+J^{(m-1)}_{-d}=
AX^d/\Pi(M,d,m-1)X^d=(W/J(M))_{-d}.\]
It follows that $J^{(0)}+J^{(1)}+\dots +J^{(m-1)}=W/J(M)$.

Also, if $s<m-1$ then \[J^{(s+1)}_0=\Pi(M,0,\widehat{s+1},m-1)/\Pi(M,0,m-1))\] and if $0<d<m$ then
\[J^{(s+1)}_d=\Pi(M,0,\widehat{s+1},m-d-1)Y^d/\Pi(M,0,m-1)Y^d\] and  \[J^{(s+1)}_{-d}=\Pi(M,d,\widehat{s+1},m-1)X^d/\Pi(M,d,m-1)X^d.\]

Thus, using Lemma~\ref{sumsandcaps}(iii), \[(J^{(0)}_d+J^{(1)}_d+\dots+J^{(s)}_d)\cap J^{(s+1)}_d=0\] for all $d$ and so
\[(J^{(0)}+J^{(1)}+\dots+J^{(s)})\cap J^{(s+1)}=0,\] whence the sum $J^{(0)}+J^{(1)}+\dots+J^{(m-1)}$ is direct.
\end{proof}

\begin{lemma}\label{uniform}
For $0\leq r\leq m-1$, the right ideal $J^{(r)}$ of $W$ is uniform.
\end{lemma}

\begin{proof} First consider the $A$-module $J^{(r)}_0$. Using Lemma~\ref{sumsandcaps},
\begin{eqnarray*}
J^{(r)}_0&=&\Pi(M,0,\widehat{r},m-1)/\Pi(M,0,m-1)\\
&=&\Pi(M,0,\widehat{r},m-1)/M_r\cap\Pi(M,0,\widehat{r},m-1)\\
&\simeq&(M_r+\Pi(M,0,\widehat{r},m-1))/M_r\\
&=&A/M_r.
\end{eqnarray*}
As $A/M$ and $A/M_r$ are isomorphic rings, $A/M_r$ is a right Ore domain and hence $J^{(r)}_0$ is a uniform right $A$-module.

We next show that if $0\neq j\in J^{(r)}$ then there exists $w\in W$ such that
$jw$ has non-zero component in degree $0$.

Let $d>0$ be such that $J^{(r)}_d\neq 0$. Thus $d\leq m-r-1$. Let \[h=aY^d+\Pi(M,0,m-d-1)Y^d\in J^{(r)}_d,\] where $a\in \Pi(M,0,\widehat{r},m-d-1)$,
and suppose that, in $J^{(r)}_{d-1}$, $hX=0$. As $aY^dX=aY^{d-1}\alpha(u)=a\alpha^d(u)Y^{d-1}$, \[0=hX=a\alpha^d(u)Y^{d-1}+\Pi(M,0,m-d)Y^{d-1}=0.\] Hence $a\alpha^d(u)\in M_r$ so either $a\in M_r$ or $\alpha^d(u)\in M_r$.
 But $0<d+r<m$ and $\alpha^i(u)\notin M$ for $0<i<m$ so $\alpha^d(u)\notin M_r=\alpha^{-r}(M)$. Therefore $a\in M_r$ so \[a\in M_r\cap \Pi(M,0,\widehat{r},m-d-1)=\Pi(M,0,m-d-1),\] by Lemma~\ref{sumsandcaps}(iii), and $h=0$.
It follows that if $0\neq h\in J^{(r)}_d$ then $0\neq hX^d\in J^{(r)}_0$.
A similar argument shows that if $0\neq h\in J^{(r)}_{-d}$ then $0\neq hY^d\in J^{(r)}_0$.
Therefore if $0\neq j\in J^{(r)}$ then there exists $w\in W$ such that
$jw$ has non-zero component in degree $0$.

Now let
 \[t=\alpha^{-(r+m-1)}(u)\dots\alpha^{-(r+1)}(u)\alpha^{-(r-1)}(u)\dots \alpha^{-(r-m+1)}(u).\] We shall see that, with $j$ and $w$ as above, $0\neq jwt\in J^{(r)}_0$.
Let $d>0$ and, as above, let \[h=aY^d+\Pi(M,0,m-d-1)Y^d\in J^{(r)}_d,\] where $a\in \Pi(M,0,\widehat{r},m-d-1)$. Then \[h\alpha^{-(r+d)}(u)=a\alpha^{-r}(u)Y^d+\Pi(M,0,m-d-1)Y^d=0,\] as
$\alpha^{-r}(u)\in M_r$ whence \[a\alpha^{-r}(u)\in M_r\cap \Pi(M,0,\widehat{r},m-d-1)=\Pi(M,0,m-d-1).\] Thus
$J^{(r)}_d\alpha^{-(r+d)}(u)=0$. Similarly, $J^{(r)}_{-d}\alpha^{d-r}(u)=0$. It follows that $J^{(r)}t\subseteq J^{(r)}_0$.

Let $h=a+\Pi(M,0,m-1)\in J^{(r)}_0$, where
$a\in \Pi(M,0,\widehat{r},m-1)$. Suppose that $ht=0$. Then $at\in M_r$ so $M_r$ contains one of
$a$, $\alpha^{-(r+m-1)}(u)$,$\dots$, $\alpha^{-(r+1)}(u)$, $\alpha^{-(r-1)}(u)$, $\dots$, $\alpha^{-(r-m+1)}(u)$. But the only integers $\ell$ such that $\alpha^{-\ell}(u)\in M_r$ are $r$ and $m+r$ so $a\in M_r$ and $h=0$.

Combining the previous three paragraphs, if $0\neq j\in J^{(r)}$ then there exists $w\in W$ such that $jw$ has non-zero component in degree $0$,
$jwt$ is homogeneous of degree $0$  and
$jwt\neq 0.$

Finally, let $j_1, j_2\in J^{(r)}\backslash\{0\}$. By the above, there exist $v_1, v_2\in W$ such that $j_1v_1$ and $j_2v_2$ are non-zero and homogeneous of degree $0$.
As $J^{(r)}_0$ is a uniform right $A$-module, it follows that
$j_1W\cap j_2W\neq 0$ and hence that $J^{(r)}$ is a uniform right $W$-module.
\end{proof}

\begin{prop}\label{Grank}
Let $W=W(A,\alpha,u)$ be a generalized Weyl algebra, with $u$ central and regular in $A$, such that, for some fixed $m \in \mathbb{N}$,
$M:=Au+A\alpha^m(u)$ is such that $A/M$ is a simple right Ore domain and
 $Au+A\alpha^i(u)=A$ for $i\in\N\backslash\{m\}$. Then the ring $W/J(M)$ has right Goldie rank $m$.
\end{prop}
\begin{proof}
This is immediate from Lemmas~\ref{Jsum} and \ref{uniform}.
\end{proof}

The next result amends Corollary~\ref{unique} to include the information on Goldie rank given by Proposition~\ref{Grank}.
\begin{cor}\label{unique2}
Let $R$ be a conformal ambiskew polynomial ring of the form $R(A,\alpha,u-\alpha(u),1)$, where $u\in Z(A)\backslash \K$ and the $\K$-algebra $A$ is a domain such that $A[y^{\pm 1},\alpha]$ is simple and $Z(A[y^{\pm 1};\alpha])=\K$. Let $\lambda\in \K$ be such that the ideal $(z-\lambda)R$ is not maximal and $(u+\lambda)A+\alpha^m(u+\lambda)A\neq A$ for some $m\geq 1$. If the factor $A/(u+\lambda)A+\alpha^m(u+\lambda)A$ is a simple  right Ore domain and $(u+\lambda)A+\alpha^n(u+\lambda)A=A$ for all $n\in \N\backslash\{m\}$  then $R/(z-\lambda)R$ has a unique non-zero prime ideal $P/(z-\lambda)R$
and $R/P$ has right Goldie rank $m$.
\end{cor}
\begin{proof}
This is immediate from Theorem~\ref{YmorXmthm} using the isomorphism between $R/(z-\lambda)R$ and $W(A,\alpha,u+\lambda)$.
\end{proof}

\begin{cor}\label{GrankS}
Suppose that $q$ is not a root of unity. Let $R=R(A,\alpha,v,1)$ be  as in Examples ~\ref{quantum torus} and \ref{spectrumSC}.
  Let $m\in \N$. The prime ideals $F_{m,1}$ and $F_{m,-1}$ of $R$ specified in Example~\ref{spectrumSC} have right Goldie rank $m$.
\end{cor}
\begin{proof}
The conditions of Corollary~\ref{unique2} are satisfied by Lemma~\ref{countableex} and the fact that $A$ is right Noetherian.
\end{proof}

\end{document}